\documentclass[english,11pt]{amsart}

\usepackage[english]{babel}

\usepackage{amscd,amssymb,amsfonts,amsmath}
\usepackage{tikz-cd}

\usepackage{bm}
\usepackage{graphics}
\usepackage{epsfig}
\usepackage{mathrsfs}
\usepackage{xcolor}
\DeclareMathAlphabet{\mathscrbf}{OMS}{mdugm}{b}{n}

\definecolor{violet}{rgb}{0.0,0.2,0.7}
\definecolor{rouge2}{rgb}{0.8,0.0,0.2}
\usepackage{hyperref}
\hypersetup{
    unicode=false,          
    pdftoolbar=true,        
    pdfmenubar=true,        
    pdffitwindow=false,     
    pdfstartview={FitH},    
    pdftitle={},    
    pdfauthor={},     
    colorlinks=true,       
   linkcolor=violet,          
    citecolor=rouge2,        
    filecolor=black,      
    urlcolor=cyan}           

\usepackage[text={6.7in,9.2in},centering]{geometry}

\makeatletter
\renewcommand\subsection{\@startsection{subsection}{2}%
  \z@{.5\linespacing\@plus.7\linespacing}{-.5em}%
  {\normalfont\sffamily}}  
\makeatother

\renewcommand{\phi}{\varphi}

\newcommand{\map}{\dashrightarrow}

\newcommand{\wb}{\overline}

\renewcommand{\ge}{\geqslant}

\newcommand{\sE}{\mathscr{E}}

\newcommand{\sG}{\mathscr{G}}
\newcommand{\sH}{\mathscr{H}}
\newcommand{\sI}{\mathscr{I}}

\newcommand{\sL}{\mathscr{L}}

\newcommand{\sO}{\mathscr{O}}

\newtheorem{thm}{Theorem}[section]

\newtheorem{lemma}[thm]{Lemma}
\newtheorem{cor}[thm]{Corollary}

\newtheorem{prop}[thm]{Proposition}

\newtheorem*{thm*}{Theorem}
\theoremstyle{definition}

\newtheorem{defn-thm}[thm]{Definition-Theorem} 
\newtheorem{defn-lemma}[thm]{Definition-Lemma}

\theoremstyle{remark}
\newtheorem{claim}[thm]{Claim}

\newtheorem*{not-and-def}{Notation and definitions}
 
\newtheorem{rem}[thm]{Remark}

\numberwithin{equation}{section}

\def\factor#1.#2.{\left. \raise 2pt\hbox{$#1$} \right/\hskip -2pt\raise -2pt\hbox{$#2$}}

\begin{document} 

\title[On foliations with semi-positive anti-canonical bundle]{On foliations with semi-positive anti-canonical bundle}

\author{St\'ephane \textsc{Druel}}

\address{St\'ephane Druel: Univ Lyon, CNRS, Universit\'e Claude Bernard Lyon 1, UMR 5208, Institut Camille Jordan, F-69622 Villeurbanne, France.} 

\email{stephane.druel@math.cnrs.fr}


\subjclass[2010]{37F75}

\begin{abstract}
In this note, we describe the structure of regular foliations with semi-positive anti-canonical bundle on smooth projective varieties. 
\end{abstract}

\maketitle
{\small\tableofcontents}

\section{Introduction}

The purpose of this paper is to prove the following result that reduces the study of regular foliations with semi-positive anti-canonical bundle to the study of regular foliations with numerically trivial canonical class.

\begin{thm}\label{thm_intro:main}
Let $X$ be a complex projective manifold, and let $\sG$ be a foliation on $X$. Suppose that $\sG$ is regular, or that $\sG$ has a compact leaf. Suppose in addition that $-K_\sG$ is semi-positive. Then there exist a smooth morphism $\phi\colon X \to Y$ onto a complex projective manifold $Y$ and a foliation $\sE$ on $Y$ with $K_\sE \equiv 0$ such that $\sG=\phi^{-1}\sE$. \end{thm}

In \cite{touzet}, Touzet obtained a foliated version of the Beauville-Bogomolov decomposition theorem
for codimension one regular foliations with numerically trivial canonical bundle on compact K\"ahler manifolds.
Pereira and Touzet then addressed regular foliations
$\sG$ with $c_1(\sG)=0$ and $c_2(\sG)=0$ on complex projective manifolds in \cite{pereira_touzet}. The structure of codimension two foliations with numerically trivial canonical bundle on complex projective uniruled manifolds is given in \cite{druel_bbcd2}.

\medskip

As a consequence of Theorem \ref{thm_intro:main}, we describe codimension one regular foliations with semi-positive anti-canonical bundle on complex projective manifolds.  

\begin{cor}\label{cor_intro:main}
Let $X$ be a complex projective manifold, and let $\sG$ be a codimension one foliation on $X$. Suppose that $\sG$ is regular, and that $-K_\sG$ is semi-positive. Then there exists a smooth morphism $\psi \colon Z \to Y$ of complex projective manifolds, as well as a finite \'etale cover $f \colon Z \to X$ and a foliation $\sE$ on $Y$ such that $f^{-1}\sG=\psi^{-1}\sE$. In addition, either $\dim Y =1$ and $\sG$ is algebraically integrable, or $Y$ is an abelian variety and $\sE$ is a linear foliation, or $Y$ is a $\mathbb{P}^1$-bundle over an abelian variety $A$ and $\sE$ is a flat Ehresmann connection on $Y \to A$.
\end{cor}

We also prove that foliations with semi-positive anti-canonical bundle having a compact leaf are automatically regular.

\begin{cor}\label{cor_intro:main2}
Let $X$ be a complex projective manifold, and let $\sG$ be a foliation on $X$. Suppose that $\sG$ has a compact leaf, and that $-K_\sG$ is semi-positive. Then $\sG$ is regular.
\end{cor}

\subsection*{Outline of the proof}The main steps for the proof of Theorem \ref{thm_intro:main} are as follows. 
In the setup of Theorem \ref{thm_intro:main}, one readily checks that the general leaves of the algebraic part of $\sG$ (we refer to Section 1 for this notion) are compact. We then consider the induced relative maximal rationally connected fibration. This gives a foliation $\sH$ on $X$ whose general leaves are projective and rationally connected. Moreover, $\sG$ is the pull-back of a foliation $\sE$ on the space of leaves of $\sH$. A result of Campana and P\u{a}un together with a theorem of Graber, Harris and Starr then implies that $K_\sE$ is pseudo-effective. Using Viehweg's weak positivity theorem, one then concludes that $-K_\sH \equiv -K_\sG$ so that $-K_\sH$ is semi-positive as well. Applying a criterion for regularity of foliations with semi-positive anti-canonical bundle, which is established in Section 2, we obtain that $\sH$ is regular. The holomorphic version of Reeb stability theorem then implies that $\sH$ is induced by a smooth morphism onto a projective manifold, finishing the proof of Theorem \ref{thm_intro:main}.

\subsection*{Acknowledgements} 
We  would  like  to thank Andreas H\"oring for useful discussions.
The author was partially supported by the ALKAGE project (ERC grant Nr 670846, 2015$-$2020)
and the Foliage project (ANR grant Nr ANR-16-CE40-0008-01, 2017$-$2020).

\section{Foliations}

In this section, we have gathered a number of results and
facts concerning foliations which will later be used in the proofs.

\subsection{Definitions}

A \emph{foliation} on  a normal complex variety $X$ is a coherent subsheaf $\sG\subseteq T_X$ such that
\begin{enumerate}
\item $\sG$ is closed under the Lie bracket, and
\item $\sG$ is saturated in $T_X$. In other words, the quotient $T_X/\sG$ is torsion-free.
\end{enumerate}

The \emph{rank} $r$ of $\sG$ is the generic rank of $\sG$.
The \emph{codimension} of $\sG$ is defined as $q:=\dim X-r$. 

\medskip

The \textit{canonical class} $K_{\sG}$ of $\sG$ is any Weil divisor on $X$ such that  $\sO_X(-K_{\sG})\cong \det\sG$. 

\medskip

Let $X^\circ \subset X_{\textup{reg}}$ be the open set where $\sG_{|X_{\textup{reg}}}$ is a subbundle of $T_{X_{\textup{reg}}}$. 
A \emph{leaf} of $\sG$ is a maximal connected and immersed holomorphic submanifold $L \subset X^\circ$ such that
$T_L=\sG_{|L}$. A leaf is called \emph{algebraic} if it is open in its Zariski closure.

The foliation $\sG$ is said to be \emph{algebraically integrable} if its leaves are algebraic.

\subsection{Foliations described as pull-backs} \label{pullback_foliations}

Let $X$ and $Y$ be normal complex varieties, and let $\varphi\colon X\map Y$ be a dominant rational map that restricts to a morphism $\varphi^\circ\colon X^\circ\to Y^\circ$,
where $X^\circ\subset X$ and  $Y^\circ\subset Y$ are smooth open subsets.
Let $\sG$ be a foliation on $Y$. \emph{The pull-back $\varphi^{-1}\sG$ of $\sG$ via $\varphi$} is the foliation on $X$ whose restriction to $X^\circ$ is $(d\phi^\circ)^{-1}\big(\sG_{|Y^\circ}\big)$.

\subsection{The family of leaves}\label{family_leaves} We refer the reader to \cite[Remark 3.12]{codim_1_del_pezzo_fols} for a more detailed explanation.
Let $X$ be a normal complex projective variety, and let $\sG$ be an algebraically integrable foliation on $X$.
There is a unique normal complex projective variety $Y$ contained in the normalization 
of the Chow variety of $X$ 
whose general point parametrizes the closure of a general leaf of $\sG$
(viewed as a reduced and irreducible cycle in $X$).
Let $Z \to Y\times X$ denotes the normalization of the universal cycle.
It comes with morphisms

\begin{center}
\begin{tikzcd}
Z \ar[r, "\beta"]\ar[d, "\psi"] & X \\
 Y &
\end{tikzcd}
\end{center}

\noindent where $\beta\colon Z\to X$ is birational and, for a general point $y\in Y$, 
$\beta\big(\psi^{-1}(y)\big) \subseteq X$ is the closure of a leaf of $\sG$.
The morphism $Z \to Y$ is called the \emph{family of leaves} and $Y$ is called
the \emph{space of leaves} of $\sG$.

\subsection{Algebraic and transcendental parts}
Next, we define the \emph{algebraic} and \emph{transcendental} parts of
a foliation (see \cite[Definition 2]{codim_1_del_pezzo_fols}).

Let $\sG$ be a foliation on a normal variety $X$.
There exist  a normal variety $Y$, unique up to birational equivalence,  a dominant rational map with connected fibers $\varphi:X\map Y$, and a foliation $\sE$ on $Y$ such that the following holds:
\begin{enumerate}
	\item $\sE$ is purely transcendental, i.e., there is no positive-dimensional algebraic subvariety through a general point of $Y$ that is tangent to $\sE$; and
	\item $\sG$ is the pullback of $\sE$ via $\varphi$.
\end{enumerate}
The foliation on $X$ induced by $\varphi$ is called the  \emph{algebraic part} of $\sG$.

\medskip

The following result due to Campana and P\u{a}un will prove to be crucial for the proof of Theorem \ref{thm_intro:main}. 
Let $X$ be a normal projective variety, and let $A$ be an ample divisor on $X$.
Recall that a $\mathbb{Q}$-divisor $D$ on a normal projective variety is said to be \textit{pseudo-effective} if,  for any positive number $\varepsilon\in\mathbb{Q}$, there exists an effective $\mathbb{Q}$-divisor $D_\varepsilon$ such that
$D+\varepsilon A\sim_\mathbb{Q} D_\varepsilon$.

\begin{thm}\label{thm:campana_paun}
Let $X$ be a normal complex projective variety, and let $\sG$ be a foliation on $X$. If $K_\sG$ is not pseudo-effective, then $\sG$ is uniruled.
\end{thm}

\begin{proof}
This follows easily from \cite[Theorem 4.7]{campana_paun15} applied to the pull-back of $\sG$ on a resolution of $X$.
\end{proof}

\section{A criterion for regularity}

Let $\sG$ be a foliation with numerically trivial canonical class on a complex projective manifold, and assume that $\sG$
has a compact leaf. Then Theorem 5.6 in \cite{lpt} asserts that
$\sG$ is regular and that there exists a foliation on
$X$ transverse to $\sG$ at any point in $X$.  In this section, we extend
this result to our setting (see also \cite[Proposition 2.7.1]{dps5} and \cite[Corollary 7.22]{fol_zero_cc}).

\begin{prop}\label{prop:criterion_regularity}
Let $X$ be a compact K\"ahler manifold of dimension $n$, and let $\sG$ be a foliation of rank $r$ on $X$. Suppose that $\sG$ has a compact leaf and that $-K_\sG$ is semi-positive. Then there is a decomposition $T_X \cong \sG\oplus \sE$ of $T_X$ into subbundles. In particular, $\sG$ is regular.
\end{prop}

\begin{proof}
Let $\omega$ be a K\"ahler form on $X$, and let $v \in H^0\big(X,\wedge^rT_X\otimes\sO_X(K_\sG)\big)$ be a $r$-vector defining $\sG$. The contraction $v \lrcorner \,\omega^r$ of $\omega^r$ by $v$ is a $\wb{\partial}$-closed $(0,r)$-form with values in 
$\sO_X(K_\sG)$, and gives a class $$[v \lrcorner \,\omega^r] \in H^r\big(X,\sO_X(K_\sG)\big).$$

We first show that $[v \lrcorner \,\omega^r]\neq 0$. Let $F \subseteq X$ be a compact leaf. We have a commutative diagram

\begin{center}
\begin{tikzcd}
H^r(X,\Omega_X^r) \ar[r, "{v \lrcorner \,\bullet}"]\ar[d] & H^r\big(X,\sO_X(K_\sG)\big) \ar[d]\\
H^r(F,\Omega_F^r) \ar[r]  & H^r\big(F,\sO_X(K_\sG)_{|F}\big).
\end{tikzcd}
\end{center}
On the other hand, since $F$ is a leaf of $\sG$, we have $\sO_F(K_F)\cong \sO_X(K_\sG)_{|F}$ and
the map $$H^r(F,\Omega_F^r) \to H^r\big(F,\sO_X(K_\sG)_{|F}\big)$$ is an isomorphism. This immediately implies that $[v \lrcorner \,\omega^r]\neq 0$.

By Serre duality, there is a $\wb{\partial}$-closed $(n,n-r)$-form $\alpha$ with values in 
$\sO_X(-K_\sG)$ such that $$\int_X\alpha\wedge (v \lrcorner \,\omega^r)=1.$$ Applying \cite[Theorem 0.1]{dps5} and using the assumption that $-K_\sG$ is semi-positive, we see that there exists a 
$\wb{\partial}$-closed $(r,0)$-form $\beta$ with values in 
$\sO_X(-K_\sG)$ such that $$\alpha = \omega^{n-r}\wedge \beta.$$ Now, a straightforward computation shows that 
$$\alpha\wedge (v \lrcorner \,\omega^r)=\omega^{n-r}\wedge \beta\wedge(v \lrcorner \,\omega^r)=C\beta(v)\omega^n$$
for some complex number $C \neq 0$. Note that $\beta(v)$ is a holomorphic function, hence constant. We conclude that $\beta(v)$
is non-zero.
It follows that the kernel of the morphism $T_X \to \Omega_X^{r-1}\otimes\sO_X(-K_\sG)$ given by the contraction with $\beta$
is a vector bundle $\sE$ on $X$ such that $T_X \cong \sG\oplus \sE$. 
This finishes the proof of the proposition.
\end{proof}

\begin{rem}
The conclusion of Proposition \ref{prop:criterion_regularity} holds if $-K_\sG$ is only assumed to be equipped with a (possibly singular) hermitian metric $h$ with curvature $i\Theta_{-K_\sG,h} \ge 0$ in the sense of currents whose multiplier ideal sheaf $\sI(h)$ is trivial, $\sI(h)\cong \sO_X$.
\end{rem}

\section{Proofs}

The present section is devoted to the proof of Theorem \ref{thm_intro:main} and Corollaries \ref{cor_intro:main} and \ref{cor_intro:main2}. Theorem \ref{thm_intro:main} is an immediate consequence of Theorem \ref{thm:main} below together with Lemma \ref{lemma:compact_leaf}.

\begin{lemma}[{\cite[Lemma 6.2]{druel15}}]\label{lemma:compact_leaf}
Let $X$ be a complex projective manifold, and let $\sG$
be a foliation on $X$. Suppose that $\sG$ is regular, or that it has a compact leaf. Then the
algebraic part of $\sG$
has a compact leaf.
\end{lemma}

\begin{thm}\label{thm:main}
Let $X$ be a complex projective manifold, and let $\sG$ be a foliation on $X$. Suppose that the algebraic part of $\sG$ has a compact leaf and that $-K_\sG$ is semi-positive. Then there exist a smooth morphism $\phi\colon X \to Y$ onto a complex projective manifold $Y$ and a foliation $\sE$ on $Y$ with $K_\sE \equiv 0$ such that $\sG=\phi^{-1}\sE$.
\end{thm}

\begin{proof}By assumption, the algebraic part of $\sG$ is induced by an almost proper map.
Its relative maximal rationally connected fibration then defines a foliation $\sH$ on $X$ with a compact leaf.

\begin{claim}
We have $K_\sH \equiv K_\sG$. 
\end{claim}

\begin{proof}
The proof is similar to that of \cite[Proposition 6.1]{druel15} (see also \cite[Proposition 8.1]{bobo} for a somewhat related result).

Let $\psi\colon Z \to Y$ be the family of leaves, and let $\beta\colon Z \to X$ be the natural morphism. By construction, 
$\phi:=\psi \circ\beta^{-1}$ is an almost proper map. Moreover, there is a foliation $\sE$ on $Y$ such that $\sG=\phi^{-1}\sE$
by \cite[Lemma 6.7]{fano_fols}.
By \cite{ghs03}, there is no rational curve tangent to $\sE$ passing through a general point of $Y$.

Let $A$ be an ample divisor on $X$, and
let $F$ be a general (smooth) fiber of $\phi$.
There exists an open set $U \supset F$ such that 
$\sH_{|U}$ is a subbundle of $\sG_{|U}$, $\phi^*\sE_{|U}$ is locally free,
and ${\big(\sG/\sH\big)}_{|U}\cong {\phi^*\sE}_{|U}$.
In particular, we have $${K_\sG}_{|F}\sim_\mathbb{Z} {K_\sH}_{|F}\sim_\mathbb{Z} K_F$$ by the adjunction formula.
Let $\varepsilon > 0$ be a rational number. Since $-K_\sG+\varepsilon A$ is $\mathbb{Q}$-ample, there exists
an effective $\mathbb{Q}$-divisor $A_\varepsilon$ such that $A_\varepsilon \sim_\mathbb{Q}-K_\sG+\varepsilon A$.
Then $K_{F}-{K_\sG}_{|F}+\varepsilon A_{|F}\sim_\mathbb{Q} {A_\varepsilon}_{|F}$
is effective. 
By Proposition \cite[Proposition 4.1]{druel15} applied to  
$\sH$ and $L:=A_\varepsilon$, we conclude
that $K_\sH + A_\varepsilon \sim_\mathbb{Q} K_\sH-K_\sG+\varepsilon A$ is pseudo-effective for any positive rational number $\varepsilon >0$. It follows that $K_\sH-K_\sG$ is pseudo-effective as well.

By \cite[Section 2.9]{druel15}, there is an effective divisor $R$ on $X$ such that
$$K_\sH-K_\sG=-(\phi^*K_\sE+R).$$
On the other hand, $K_\sE$ is pseudo-effective by Theorem \ref{thm:campana_paun}. This easily implies that $\phi^*K_\sE$ is pseudo-effective as well. Therefore, we must have $K_\sH-K_\sG\equiv 0$, proving our claim.
\end{proof}

By the claim, there is a flat line bundle $\sL$ on $X$ such that $\sO_X(K_\sH) \cong \sO_X(K_\sG)\otimes \sL$. On the other hand, $\sL$ admits a unitary smooth hermitian metric with zero curvature. It follows that $-K_\sH$ is semi-positive as well.
By Proposition \ref{prop:criterion_regularity} above, there is a decomposition $T_X \cong \sH\oplus \sH_1$ of $T_X$ into subbundles. In particular, $\sH$ is regular. By construction, its leaves are projective rationally connected manifolds. In particular, they are simply connected. Arguing as in the proof of \cite[Lemma 4.1]{druel_fol_fano}, we see that 
$\beta$ is an isomorphism and that $\psi$ is a smooth morphism onto a complex projective manifold. 
A straightforward computation then shows that $K_\sE \equiv 0$ since $K_\sH\equiv K_\sG$.
This finishes the proof of the theorem.
\end{proof}

\begin{proof}[Proof of Corollary \ref{cor_intro:main}]
The statement follows from Theorem \ref{thm_intro:main} together with \cite[Th\'eor\`eme 1.2]{touzet} using \cite[Lemma 5.9]{bobo}
\end{proof}

\begin{proof}[Proof of Corollary \ref{cor_intro:main2}]
The claim follows from Theorem \ref{thm_intro:main} and \cite[Theorem 5.6]{lpt} (or Proposition \ref{prop:criterion_regularity}).
\end{proof}

\providecommand{\bysame}{\leavevmode\hbox to3em{\hrulefill}\thinspace}
\providecommand{\MR}{\relax\ifhmode\unskip\space\fi MR }
\providecommand{\MRhref}[2]{%
  \href{http://www.ams.org/mathscinet-getitem?mr=#1}{#2}
}
\providecommand{\href}[2]{#2}


\end{document}